\documentclass[12pt,a4paper]{article}

\usepackage{amsmath, amsthm, amsfonts, amssymb,color,geometry,enumerate,color}
\usepackage{graphicx}
\usepackage{fancybox}
\usepackage{cases}
\usepackage{fancyhdr}

\theoremstyle{definition}

\newtheorem{theorem}{Theorem}[section]
\newtheorem{proposition}[theorem]{Proposition}
\newtheorem{lemma}[theorem]{Lemma}
\newtheorem{definition}[theorem]{Definition}
\newtheorem{remark}[theorem]{Remark}
\newtheorem{problem}[theorem]{Problem}
\newtheorem{example}[theorem]{Example}

\makeatletter

\@addtoreset{equation}{section}
\makeatother

\newcommand{\R}{\mathbb{R}}   
\newcommand{\C}{\mathbb{C}}   
\newcommand{\N}{\mathbb{N}}   
\renewcommand{\epsilon}{\varepsilon}    
\newcommand{\supp}{{\rm supp}}   

\begin{document}

\title{Log-unimodality for free positive multiplicative Brownian motion}
\author{\Large{Takahiro Hasebe, Yuki Ueda, Jiun-Chau Wang}}
\date{}

\maketitle

\abstract{We prove that the marginal law $\sigma_{t}\boxtimes\nu$ of free positive multiplicative Brownian motion is log-unimodal for all $t>0$ if $\nu$ is a multiplicatively symmetric
log-unimodal distribution, and that $\sigma_{t}\boxtimes\nu$ is log-unimodal
for sufficiently large $t$ if $\nu$ is supported on a suitably chosen
finite interval. Counterexamples are given when $\nu$ is not assumed to be symmetric
or having a bounded support.}




\section{Introduction}


This paper is a continuation of the first two authors' works \cite{HU18, HU}
on the unimodality of free Brownian motions.

Since its first appearance in \cite{Bia97a}, free multiplicative Brownian
motion has been an object of interest in free probability. For examples, Biane showed that free unitary multiplicative Brownian
motion can be approximated by $\text{U}_N$-valued Brownian motion as $N$ tends to infinity, and he calculated the moments and the density of the marginal laws of free unitary multiplicative Brownian motion in \cite{Bia97a} and \cite{Bia97b}. 
Notably, Kemp \cite{Kemp} and C\'{e}bron \cite{Cebron} introduced another type of free multiplicative Brownian motion approximated by $\text{GL}_N$-valued Brownian motion as the size $N$ of matrices tends to infinity.

Biane defined in \cite[Definition 4.2]{Bia98} free positive multiplicative increment processes, which contain as a special case free positive multiplicative Brownian motion. Given a probability measure $\nu$ on $\R_+:=(0,\infty)$,
the free positive multiplicative Brownian motion with initial distribution
$\nu$ has the marginal laws $\left\{ \sigma_{t}\boxtimes\nu:t\ge0\right\} $, where 
 $\sigma_{t}$ is the $\boxtimes$-infinitely divisible measure whose $\Sigma$-transform
is given by 
\[
\Sigma_{\sigma_{t}}(z)=\exp\left(\frac{t}{2} \cdot \frac{z+1}{z-1}\right).
\]
Zhong proved in \cite{Zhong} 
that $\sigma_{t}\boxtimes\nu$ is absolutely continuous with a continuous density relative to the Lebesgue measure on $\R_+$; Section \ref{sec2.2} of the present paper summarizes the results.

In this paper, we address the unimodality of $\sigma_{t}\boxtimes\nu$.
Recall that a positive Borel measure $\mu$ on $\mathbb{R}$ is said
to be \emph{unimodal with mode $a$ }if $\mu=c\delta_{a}+f(x)\,dx$
where $c\in[0,\infty]$ and $f:\mathbb{R}\rightarrow[0,\infty)$ is non-decreasing
on $(-\infty,a)$ and non-increasing on $(a,\infty)$. In view of
the analytic apparatus of free multiplicative convolution, we find
that it is more natural and appropriate to study the unimodality of
the measure $x\,d(\sigma_{t}\boxtimes\nu)(x)$ than that of $\sigma_{t}\boxtimes\nu$
itself. It turns out that the unimodality of $x\,d(\sigma_{t}\boxtimes\nu)(x)$
is equivalent to the unimodality of the push-forward measure $\log_{*}(\sigma_{t}\boxtimes\nu)$
by the logarithmic function $\log\colon \R_+ \rightarrow\mathbb{R}$.
We refer to Section \ref{sec3} for a detailed discussion of this \emph{log-unimodality}.

Our main results are as follows. We prove in Theorem \ref{thm:symmlogunimo} that $\sigma_{t}\boxtimes\nu$
is log-unimodal for every $t>0$ if the initial distribution $\nu$
is log-unimodal and symmetric about $1$ with respect to the multiplication
on $\R_+$. In particular, $\sigma_{t}$ itself is log-unimodal
for all $t>0$. In Theorem 4.4, we show that if $\nu$ is supported
on a closed interval $[\alpha,\beta]$ where $\beta^{4}<2\alpha^{3}\beta+3\alpha^{4}$,
then the process $\sigma_{t}\boxtimes\nu$ becomes log-unimodal for
sufficiently large $t$. The log-unimodality may fail when $\nu$ is not symmetric or if $\nu$ has an unbounded support, see Theorem \ref{thm:non-log} and Example \ref{ex:unbounded_support}. 

The paper is organized as follows. After recalling basic results of
free convolution in Section \ref{sec2}, we introduce and investigate the class
of log-unimodal measures in Section \ref{sec3}. The main results are
proved in Section \ref{sec4}. 


\section{Preliminaries}\label{sec2}

\subsection{Free multiplicative convolution}
The free multiplicative convolution $\mu \boxtimes \nu$ of probability measures $\mu$ and $\nu$ on $\mathbb{R}_{+}$ is defined to be the distribution of $X^{\frac1{2}} Y X^{\frac1{2}}$, where $X, Y$ are free independent, non-negative self-adjoint operators affiliated with a finite von Neumann algebra and having the distributions $\mu$ and $\nu$, respectively.  

For a probability measure $\mu$ on $\mathbb{R}_{+}$, we define
\begin{align*}
\psi_\mu(z)=\int_0^\infty \frac{xz}{1-xz}d\mu(x), \qquad z\in \C\setminus[0,+\infty),
\end{align*}
and
\begin{align*}
\eta_\mu(z)=\frac{\psi_\mu(z)}{1+\psi_\mu(z)}.
\end{align*}
It is shown in \cite{BV93} that the function $\eta_\mu$ has an analytic compositional inverse $\eta_\mu^{-1}$ defined in a neighborhood of $(-\infty,0)$. Accordingly, we define the $\Sigma$-transform of $\mu$ by 
\begin{align*}
\Sigma_\mu(z)=\frac{\eta_\mu^{-1}(z)}{z}.
\end{align*}  
The measure $\mu$ is uniquely determined by its $\Sigma$-transform.

For probability measures $\mu$ and $\nu$ on $\R_+$, their free multiplicative convolution $\mu\boxtimes \nu$ is determined by
\begin{align*}
\Sigma_{\mu\boxtimes \nu}(z)=\Sigma_\mu(z)\Sigma_\nu(z),\quad z\in (-\infty,0).
\end{align*}

\subsection{Density function of $\sigma_t\boxtimes \nu$}\label{sec2.2}

We review Zhong's density formula \cite{Zhong} as follows. Fix a probability measure $\nu$ on $\R_+$. Define a function $u_t\colon(0,\infty)\rightarrow[0,\pi)$ and a set $V_{t,\nu}$ by
\begin{align*}
u_t(r)=\inf\left\{ \theta\in(0,\pi): \frac{\sin \theta}{\theta}\int_0^\infty \frac{r\xi}{1+r^2\xi^2-2r\xi\cos \theta}d\nu(\xi)\le \frac{1}{t}\right\},\quad r>0,
\end{align*} and \begin{align*}
V_{t,\nu}=\left\{ r>0: \int_0^\infty \frac{r\xi}{(1-r\xi)^2}d\nu(\xi)>\frac{1}{t}\right\}.
\end{align*} It is shown in \cite{Zhong} that the function $u_t$ is continuous on $(0,+\infty)$ and that $r\in V_{t,\nu}$ if and only if $u_t(r)>0$; in which case, $u_t(r)$ is the unique solution $\theta\in(0,\pi)$ of the equation
\[
\frac{\sin \theta}{\theta}\int_0^\infty \frac{r\xi}{1+r^2\xi^2-2r\xi\cos \theta}d\nu(\xi)=\frac{1}{t}.
\] Moreover, the map 
\begin{align*}
\Lambda_{t,\nu}(r)=r\exp\left(\frac{t}{2}\int_0^\infty \frac{r^2\xi^2-1}{|1-r\xi e^{iu_t(r)}|^2}d\nu(\xi) \right),\quad r>0,
\end{align*}
is a homeomorphism from $(0,\infty)$ to $(0,\infty)$. 

The measure $\sigma_t\boxtimes \nu$ is Lebesgue absolutely continuous with a continuous density $q_t$ given by
\begin{align}\label{eq:Zhongformula}
xq_t\left(x \right)=\frac{u_t \left( \Lambda^{-1}_{t,\nu}(1/x) \right)}{\pi t}, \qquad x\in (0,\infty).
\end{align}
Thus, the support of $\sigma_t\boxtimes \nu$ is the closure $\overline{\{x>0:1/x\in \Lambda_{t,\nu}\left(V_{t,\nu}\right)\}}$, and the function $q_t$ is analytic on the set $\{x>0:1/x\in \Lambda_{t,\nu}\left(V_{t,\nu}\right)\}$.

\section{Log-unimodal distributions}\label{sec3}

\subsection{Characterizations of log-unimodality}

\begin{definition} A positive Borel measure $\mu$ on $\mathbb{R}_+$ is said to be {\it log-unimodal} with mode $a\in \R_+$ if its push-forward $\log_*\mu$ by the logarithm function $\log\colon\mathbb{R}_+\rightarrow\mathbb{R}$ is a unimodal measure with mode $\log a$ on $\mathbb{R}$.  
\end{definition}
\begin{lemma}\label{lem:log-unimodal}
A positive Borel measure $\mu$ on $\R_+$ is log-unimodal with mode $a \in \R_+$ if and only if the measure $x\,d\mu(x)$ is unimodal with mode $a$.
\end{lemma}
\begin{proof}
Assume that $\mu$ is log-unimodal with mode $a$. Then there are $c\ge0$ and a function $f\colon\mathbb{R} \rightarrow [0,\infty)$ which is non-decreasing on $(-\infty, \log a)$ and non-increasing on $(\log a,\infty)$ such that
\begin{align*}
(\log_*\mu)(dx)=c\delta_{\log a}+f(x)\, dx, \qquad x \in \R. 
\end{align*}
It follows that
\begin{align*}
\mu(dx)=c\delta_{a}+\frac{f(\log x)}{x}dx, \qquad x>0,
\end{align*}
and therefore,
\begin{align*}
x\mu(dx)=c a\delta_{a}+f(\log x)dx, \qquad x>0.
\end{align*}
Since the logarithm function $\log\colon\mathbb{R}_+\rightarrow\mathbb{R}$ is strictly increasing, the function $x\mapsto f(\log x)$ is non-decreasing on $(0,a)$ and non-increasing on $(a,\infty)$. It follows that $x\mu(dx)$ is unimodal with mode $a>0$. The converse implication is proved in the same way.
\end{proof}


The next result shows that if the distribution of a positive random variable $X$ is log-unimodal, then so is the distribution of its multiplicative inverse $X^{-1}$. The proof is similar to that of Lemma 3.2, and the details are left to the interested reader.
\begin{proposition}\label{prop:inverse}
Let $a>0$. Let $\mu$ be a positive Borel measure on $\R_+$ and denote $d\mu^{-1}(x)=d\mu(1/x)$. The following conditions are equivalent.
\begin{enumerate}[(1)]
\item $\mu$ is log-unimodal with mode $a$.
\item $\mu^{-1}$ is log-unimodal with mode $1/a$.
\end{enumerate}
\end{proposition}

We now prove that the class of log-unimodal probability measures is weakly closed. Note that the family of unimodal probability measures on $\R$ is weakly closed, see \cite{Sato}.
\begin{lemma}\label{lem:weak}
The set of log-unimodal probability measures is closed with respect to weak convergence of probability measures on $\R_+$.
\end{lemma}
\begin{proof}
Let $\{\mu_n\}_{n\in \N}$ be a sequence of log-unimodal probability measures on $\R_+$. Assume that there is a probability measure $\mu$ on $\R_+$ such that $\mu_n\xrightarrow{w}\mu$ as $n\rightarrow\infty$. By the continuous mapping theorem, the push-forward measures $\log_*\mu_n$ converge weakly to $\log_*\mu$, and therefore the limit distribution $\log_*\mu$ is unimodal on $\R$. By definition, this means that $\mu$ is log-unimodal.
\end{proof}


We characterize log-unimodal distributions by their moment generating functions. The next result is a generalization of Isii's characterization \cite[Theorem 3.2]{Isi57} for unimodal probability measures on $\R$. 
\begin{theorem}\label{thm:gen_isii} Let $\tau$ be a positive Borel measure on $\R$ such that 
\[
\int_{\R} \frac{\tau(dx)}{1+x^2}<\infty.
\]
We define the associated Pick function 
\[
P_\tau(z) = \int_\R \frac{1+xz}{(x-z)(1+x^2)}\,\tau(dx), \qquad z\in \C^+.
\]
Then $\tau$ is unimodal with mode $c \in \R$ if and only if 
\[
\Im[(z-c) P_\tau'(z)] \le 0, \qquad z \in \C^+. 
\]
\end{theorem}
\begin{proof} The ``only if'' part is easier to prove. Thus, assume that $\tau$ is unimodal and observe that $\tau_n(dx) = 1_{[-n,n]}(x)\tau(dx)$ is a unimodal measure on $\R$ with mode $c$ for each positive integer $n>|c|$, and that 
\[
P_{\tau_n}(z) = \int_\R\left( \frac{1}{x-z} -\frac{x}{1+x^2} \right)\tau_n(dx) = - \int_\R\frac{x}{1+x^2}\,\tau_n(dx) + \int_\R\frac{1}{x-z}\,\tau_n(dx).
\]
Isii's result \cite[Theorem 3.2]{Isi57} implies that  
\[
\Im[(z-c) P_{\tau_n}'(z)] \le 0, \qquad z \in \C^+.   
\]
Since $P_{\tau_n}$ converges to $P_\tau$ pointwisely in $\C^+$, the desired inequality follows. 

The ``if'' part needs a more detailed analysis. Thus, we assume $\Im[(z-c) P_\tau'(z)] \le 0$ for all $z\in \C^+$. We also assume $c=0$; the general case follows by a translation of $\tau$ and $P_\tau$ by the amount of $c$. In what follows, the functions $\log$ and $\arg$ are defined continuously on $\C\setminus i(-\infty,0]$ such that $\log 1=0$ and $\arg 1=0$.  
Since $Q(z)=-z P_\tau'(z)$ is a Pick function defined for $z\in \C^+$, there exist $a \ge0$, $b\in \R$, and a finite Borel measure $\rho$ such that 
\begin{equation}\label{eq:pick}
Q(z) = a z + b + \int_\R\frac{1+xz}{x-z} d\rho(x),  \qquad z \in \C^+.  
\end{equation}

\noindent
{\bf Step 1.} Fix  $0<\delta< \kappa <\infty$ such that $\rho$ is continuous at $\delta$ and $\kappa$, and define
\[
F(t)= \int_{[\delta, t]} \frac{1+x^2}{x} \,d\rho(x),\qquad t \in [\delta,\kappa]. 
\]
We will prove that there exists a continuous function $E$ on $\C^+\cup(\delta,\kappa)$ such that $\Im[E(z)]$ is a constant for all $z\in(\delta,\kappa)$ and 
\begin{equation}\label{eq:decomp}
P_\tau(z) = \int_\delta^\kappa \frac{F(x)}{z-x} \, dx +E(z),  \qquad z \in \C^+.  
\end{equation}
To this end, we first notice that 
\[
P_\tau'(z) = - a - \frac{b}{z} + \int_{\R}\frac{1+xz}{z(z-x)}  d\rho(x).  
\]
Integrating the both sides from $i$ to $z$ with Fubini's theorem yields 
\[
P_\tau(z) = \gamma - a z - b \log z  + \int_{\R} k(x,z)\,   d\rho(x),  \qquad z \in \C^+, 
\] 
where $\gamma = P_\tau(i) + ai + b\log i $ and 
\begin{align*}
k(x,z) &=\int_{i}^z \frac{1+xw}{w(w-x)}\,dw, \qquad x \in \R,\, z \in (\C^+\cup \R)\setminus\{0\},\, x\neq z,\\
&=  \int_i^z \left(- \frac{1}{xw} + \frac{1+x^2}{x(w-x)}\right) dw \\ 
&= \frac{1+x^2}{x}\Bigg[ \underbrace{ \log(z-x)- \log(i-x)-\frac{1}{1+x^2} (\log z - \log i)}_{\text{denoted by} ~\ell(x,z)}\Bigg] \qquad \text{if}  \qquad x\neq 0. 
\end{align*}
 Now $P_\tau$ is of the form
\[
P_\tau(z) = \int_\delta^\kappa k(x,z)\,  d\rho(x) + E_0(z),   \qquad z \in \C^+, 
\]
where $E_0$ is the  continuous function on $\C^+\cup(\delta,\kappa)$ defined by   
\[
E_0(z) = \gamma - a z - b \log z+ \int_{(-\infty, \delta)} k(x,z)\,   d\rho(x) +\int_{(\kappa,\infty)} k(x,z)\,   d\rho(x).
\] 
By integration by parts we obtain 
\begin{align*}
\int_\delta^\kappa k(x,z)\,  d\rho(x) 
&= \int_\delta^\kappa \ell(x,z)\, dF(x) =  [\ell(x,z) F(x)]_{x=\delta}^{x=\kappa} - \int_{\delta}^\kappa \partial_x \ell(x,z)\, F(x) \,dx \\
&= \ell(\kappa,z) F(\kappa) +\int_\delta^\kappa \left[ \frac{1}{z-x} - \frac{1}{i-x} -\frac{x(2\log z - \pi i)}{(1+x^2)^2}  \right] F(x)\, dx. 
\end{align*}

In summary, the function $P_\tau$ is of the desired form \eqref{eq:decomp} 
where $E$ is defined by 
\[
E(z) = E_0(z) + \ell(\kappa,z) F(\kappa)  - \int_\delta^\kappa \left[ \frac{1}{i-x}+ \frac{x(2\log z - \pi i)}{(1+x^2)^2}\right]F(x)\, dx. 
\]
Moreover, $\Im [E(z)]$ and $\Im [E_0(z)]$ are constant functions on $(\delta,\kappa)$, because 
\begin{equation*}\label{eq:const}
\begin{split}
\Im [\ell(x,z)] 
&= \arg(z-x)- \arg(i-x)+\frac{\pi}{2(1+x^2)}   \\
&= 
\begin{cases} 
\pi- \arg(i-x)+ \frac{\pi}{2(1+x^2)}, & x\ge \kappa, \\[1mm] 
- \arg(i-x)+ \frac{\pi}{2(1+x^2)}, &x\le \delta 
\end{cases}
\end{split}
\end{equation*}
is a constant function for $z\in (\delta, \kappa)$. 

\vspace{2mm}
\noindent
{\bf Step 2.}
We apply the Stieltjes inversion formula to \eqref{eq:decomp}. For convenience, we denote by $T$ the Stieltjes transform of the measure $F(x)1_{[\delta,\kappa]}(x)\,dx$ and write $P_\tau(z) = T(z) + E(z)$. 
For two continuity points $s,t \in (\delta,\kappa)$ of $\tau$ with $s <t$, we have 
\begin{align*}
\tau([s,t])
&=\frac{1}{\pi}\lim_{\epsilon\to0^+}\int_s^t \Im[P_\tau(x+i\epsilon)]\,dx \\
&=    \frac{1}{\pi}\lim_{\epsilon\to0^+}\int_s^t \Im[T(x+i\epsilon)]\,dx+ \frac{1}{\pi} \int_s^t \Im[E(x)]\,dx \\
&=\int_s^t (\gamma-F(x))\,dx,   
\end{align*}
where $\gamma  = \pi^{-1} \Im [E(x)]$ is a constant, possibly depending on $\delta$ and $\kappa$. 
Therefore,  $\tau$ is Lebesgue absolutely continuous on $(\delta,\kappa)$ with density $p$ given by 
\begin{equation}\label{eq:p}
p(x) = \gamma- F(x)  = \gamma -\int_\delta^x \frac{1+y^2}{y}\,d\rho(y), \qquad x \in (\delta,\kappa). 
\end{equation}
This shows that $p$ is non-increasing on $(\delta,\kappa)$. Letting $\delta \to 0$ and  $\kappa\to \infty$, we conclude that $p$ is non-increasing on $(0,\infty)$. Similarly, we can prove that $\tau$ has a non-decreasing density $p$ on $(-\infty,0)$ as well, and hence $\tau$ is unimodal with mode $0$.  
\end{proof}

\begin{remark} The formula \eqref{eq:tau} below shows a finer relationship among $\tau$, $\rho$ and $p$. To prove it, we first notice that the constant $\gamma$ is actually independent of $\kappa$ by virtue of \eqref{eq:p}. Since $p$ is non-increasing and nonnegative on $(0,\infty)$, the limit $\alpha=\lim_{x\to\infty}p(x)$ exists in $[0,\infty)$, implying further that $\int_\delta^\infty \frac{1+y^2}{y}\,d\rho(y)<\infty$. If we take $p$ to be right-continuous,  we obtain 
\begin{equation}\label{eq:density_pos}
p(x) = \alpha+ 
\int_{(x,\infty)} \frac{1+y^2}{y}\,d\rho(y), \qquad x>0.
\end{equation}
Similarly, we have
\begin{equation}\label{eq:density_neg}
p(x)=\beta+ \int_{(-\infty, x]} \frac{1+y^2}{|y|}\,d\rho(y), \qquad x < 0 
\end{equation}
for some constant $\beta\ge0$ and 
\begin{equation}
\int_{\R} |y|\,d\rho(y) <\infty. 
\end{equation} 
The relation $-z P_\tau'(z) =Q(z)$ and the dominated convergence theorem yield
\begin{equation}
\tau(\{0\})= \lim_{\epsilon\to0^+, z=i \epsilon} z^2 P_\tau'(z) =-\lim_{\epsilon\to0^+, z=i \epsilon} z Q(z) = \rho(\{0\}). 
\end{equation}
So we have 
\begin{equation}\label{eq:tau}
\tau(dx) = p(x)\,dx +\rho(\{0\}) \delta_0. 
\end{equation}

Conversely, given a finite Borel measure $\rho$ on $\R$ with $\int_{\R} |y|\,d\rho(y) <\infty$ and given two constants $\alpha,\beta\ge0$, we define a measure $\tau$ by \eqref{eq:tau} where $p$ is defined through \eqref{eq:density_pos} -- \eqref{eq:density_neg}, then $\tau$ is unimodal with mode 0 such that $\int_\R \frac{1}{1+x^2}\,d\tau(x)<\infty$.  

Applying integration by parts to $-z P_\tau'(z)$ together with formulas \eqref{eq:density_pos} -- \eqref{eq:tau} leads to 
\[
Q(z) = -z P_\tau'(z)= \alpha-\beta+ \int_\R \frac{1+x^2}{x-z} \,d\rho(x), 
\]
which means $a=0$ and $b= \alpha-\beta+ \int_\R x\,d\rho(x)$ in \eqref{eq:pick}.  
\end{remark}

\begin{proposition}
A probability measure $\mu$ on $\R_+$ is log-unimodal with mode $c \in \R_+$ if and only if the following inequality holds: 
\begin{align}\label{eq:characterization}
\Im \left[ z(1-cz)\psi_\mu'(z) \right]\ge0, \qquad z\in \C^+.
\end{align}
\end{proposition}
\begin{proof}
By Lemma \ref{lem:log-unimodal}, $\mu$ is log-unimodal with mode $c$ if and only if the measure $\tau(dx)=x\,d\mu(x)$ on $\R_+$ is unimodal with mode $c$. The result follows from Theorem \ref{thm:gen_isii} and the fact 
\begin{align*}
P_{\tau}'(z)=\int_{\R_+} \frac{x}{(z-x)^2}d\mu(x)=\frac{1}{z^2}\psi_{\mu}'\left( \frac{1}{z}\right),\quad z\in \C \setminus [0,+\infty).
\end{align*}
\end{proof}

\begin{remark} This is to be compared with a similar characterization in the unitary case \cite[Theorem 7.16]{FHS18}. 
\end{remark}

\begin{example} The following distributions are log-unimodal.
\begin{enumerate}[(1)] 
\item (Half normal distribution) If $t>0$ and $X\sim N(0,t)$, then the distribution $\rho_t$ of $|X|$ is called a {\it half normal distribution} and its density function is given by
\begin{align*}
\frac{2}{\sqrt{2\pi t}} \exp\left(-\frac{x^2}{2t} \right) \mathbf{1}_{\R_+}(x).
\end{align*}
The measure $x\rho_t(dx)$ is unimodal with mode $\sqrt{t}$. Therefore the half normal distribution $\rho_t$ is log-unimodal with mode $\sqrt{t}$.

\item (Gamma distribution) The probability measure 
\[
\gamma_{\theta,p}(dx) = \frac{\theta^{-p}}{\Gamma(p)} x^{p-1}e^{-x/\theta}\mathbf{1}_{\R_+}(x)\,dx, \qquad p,\theta>0,
\]
is log-unimodal with mode $p\theta$. Furthermore the inverse gamma distribution is also log-unimodal by Proposition \ref{prop:inverse}.

\item (Beta distribution) The probability measure
\[
\beta_{p,q}(dx)=\frac{1}{B(p,q)} x^{p-1}(1-x)^{q-1}\mathbf1_{(0,1)}(x)\,dx, \qquad p,q>0,
\]
is log-unimodal with mode $s$, where
\begin{align*}
s=\begin{cases}
\frac{p}{p+q-1}, & p>0, \hspace{2mm}q>1,\\
1, & p>0, \hspace{2mm}0<q\le 1.
\end{cases}
\end{align*}

\item (Marchenko-Pastur distribution) The probability measure
\[
\pi(dx)=\frac{1}{2\pi}\sqrt{\frac{4-x}{x}} \mathbf{1}_{(0,4]}(x)dx,
\]
is log-unimodal with mode $2$. It is known that the measure $\pi^{-1}$ is the positive free stable laws with index $1/2$. By Proposition \ref{prop:inverse}, the measure $\pi^{-1}$ is log-unimodal with mode $1/2$.

\item (Positive Boolean stable laws with index $\alpha\in (0,1)$) The probability measure
\[
b_{\alpha}(dx)=\frac{\sin \pi \alpha}{\pi}\cdot  \frac{x^{\alpha-1}}{x^{2\alpha}+2x^{\alpha}\cos \pi \alpha+1} \mathbf{1}_{(0,\infty)}(x)\,dx,
\]
is log-unimodal with mode $1$.
\end{enumerate}
\end{example}

\subsection{Log-unimodality under classical multiplicative convolution}\label{sec3.2}

For two probability measures $\mu$ and $\nu$ on $\R_+$, their classical multiplicative convolution $\mu\circledast\nu$ is the distribution of $XY$, where $X$ and $Y$ are independent positive random variables distributed according to $\mu$ and $\nu$. In this section we study the log-unimodality of measures under the convolution $\circledast$. 

We first look at symmetric measures. Recall that a Borel measure $\mu$ on $\R_+$ is said to be {\it multiplicatively symmetric} if $\mu=\mu^{-1}$. The free Brownian motion $\sigma_t$ is multiplicatively symmetric for all $t>0$.

A measure $\nu$ on $\R$ is said to be additively symmetric if the push-forward measure $\exp_*(\nu)$ by the exponential map $e^x:\R \rightarrow \R_+$ is multiplicatively symmetric. In other words, $\nu$ being additively symmetric means that the mass distribution of $\nu$ is symmetric with respect to the origin.

Since the exponential map turns the classical additive convolution $\ast$ to the multiplicative convolution $\circledast$ and the additively symmetric unimodal probability measures on $\R$ are preserved by the convolution $\ast$ (see \cite[Exercise 29.22]{Sato} or the original article of Wintner \cite[Theorem XIII]{W}), we obtain the following result.  
\begin{proposition}
If $\mu$ and $\nu$ are multiplicatively symmetric log-unimodal probability measures on $\R_+$, then $\mu\circledast\nu$ is multiplicatively symmetric and log-unimodal.
\end{proposition}

Next, we consider the so-called strong unimodality. A probability measure $\mu$ on $\R_+$ is said to be $\circledast$-{\it strongly log-unimodal} if for all log-unimodal distributions $\nu$, the convolution $\mu\circledast\nu$ is log-unimodal. The notion of strong unimodality relative to other convolutions is defined analogously. By virtue of the exponential map again, the next result follows immediately.  
\begin{proposition}\label{prop:s-l-unimodal}
A probability measure $\mu$ on $\R_+$ is $\circledast$-strongly log-unimodal if and only if $\log_*\mu$ is $*$-strongly unimodal on $\R$.
\end{proposition}

\begin{example}
The log-normal distribution is $\circledast$-strongly log-unimodal since the normal distribution is $*$-strongly unimodal.
\end{example}

For $b\in (0,\pi)$, we define the following probability measure:
\begin{align*}
\lambda_b(dx)=\frac{c_b}{1-2x\cos b+x^2}dx, \qquad x\in \R_+,
\end{align*}
where $c_b = \sin b / (\pi-b)$ is a normalizing constant such that $\lambda_b(\R_+)=1$. By examining its density directly, it is easy to see that $\lambda_b$ is multiplicatively symmetric and log-unimodal with mode $1$ for all $b\in (0,\pi)$. 


\begin{lemma}\label{lem:lambda}
The measure $\lambda_b$ is $\circledast$-strongly log-unimodal on $\R_+$ if and only if $\cos b \leq 0$.
\end{lemma}
\begin{proof}
The density function $f(x)=\frac{d\log_*\lambda_b}{dx}(x)$ is easily seen to be
\begin{align*}
f(x)=\frac{c_be^x}{1-2e^x\cos b+e^{2x}}, \qquad x\in\mathbb{R}.
\end{align*}
Define the function $g(x)=\log f(x)$. Then we have
\begin{align*}
g''(x)=\frac{2e^{x}(\cos b-2e^x+ e^{2x}\cos b)}{(1-2e^x\cos b+ e^{2x})^2}.
\end{align*}
One can see that $g''\le 0$ on $\R$ if and only if $\cos b \le 0$. Thus, the function $g$ is concave downward on $\R$ if and only if $\cos b \leq 0$. A result of Ibragimov \cite{I56} shows that the concavity of $g$ is equivalent to the $\ast$-strong unimodality of $\log_*\lambda_b$, whence the desired result follows from Proposition \ref{prop:s-l-unimodal}. \end{proof}


\section{Main results}\label{sec4}
We begin with a criterion for the log-unimodality of $\sigma_t\boxtimes \nu$. Recall that $q_t$ denotes the density function of $\sigma_t\boxtimes \nu$ in Zhong's formula.

\begin{lemma}\label{lem:freemulti}
Let $t>0.$ The following conditions are equivalent.
\begin{enumerate}[(1)]

\item The measure $\sigma_t\boxtimes \nu$ is log-unimodal.

\item\label{item:freemulti3} For each $a\in(0,1/t)$, the equation
\[\frac{r}{c_{a\pi t}}\cdot \frac{d(\lambda_{a\pi t}\circledast \nu^{-1})}{dx}(r)=\frac{a\pi}{\sin (a\pi t)}\] about $r$ has at most two solutions in $\R_+$. Recall that  the measure $\lambda_{a\pi t}$ and the normalization constant $c_{a\pi t}$ are defined in Section \ref{sec3.2} and $d\nu^{-1}(x)=d\nu(1/x)$.

\item For each $R\in (0,\pi)$, the equation
\begin{align*}
\frac{\sin R}{R}\int_0^\infty \frac{r\xi}{1+r^2\xi^2-2r\xi \cos R}d\nu(\xi)=\frac{1}{t}
\end{align*} about $r$ has at most two solutions in $\R_+$.
\end{enumerate}
\end{lemma}

\begin{proof}
This proof is similar to \cite[Lemma 3.1]{HU18}, and we present it here for the sake of completeness. 

The equivalence between (1) and (2) is obvious. 
We show that the measure $\sigma_t\boxtimes \nu$ is log-unimodal if and only if for each $a\in(0,1/t)$, the equation
\begin{align}\label{equation}
\int_0^\infty \frac{r\xi}{1+r^2\xi^2-2r\xi\cos (a\pi t)}d\nu(\xi)=\frac{a\pi}{\sin (a\pi t)}
\end{align}
in $r$ has at most two solutions in $\R_+$. The latter condition is easily seen to be equivalent to (3) from the substitution $R=a\pi t$. 

Assume that $\sigma_t\boxtimes \nu$ is log-unimodal. By Zhong's formula \eqref{eq:Zhongformula}, we have
\begin{align*}
xq_t(x)=\frac{u_t(\Lambda_{t,\nu}^{-1}(1/x))}{\pi t}.
\end{align*}
By Lemma \ref{lem:log-unimodal}, the measure $xq_t(x)dx$ is unimodal. Since the continuous density $xq_t(x)$ is analytic whenever it is positive, the graph of $xq_t(x)$ has no plateau above the real line. It follows that for each $a\in(0,1/t)$, the equation $u_t(\Lambda_{t,\nu}^{-1}(1/x))=a \pi t$
in $x$ has at most two solutions in $\R_+$, that is, the equation $u_t(r)=a \pi t$ in $r$ has at most two solutions in $\R_+$. Finally, since
\begin{align*}
\frac{\sin (u_t(r))}{u_t(r)}\int_0^\infty \frac{r\xi}{1+r^2\xi^2-2r\xi\cos(u_t(r))}d\nu(\xi)=\frac{1}{t},
\end{align*}
we conclude that the equation \eqref{equation} has at most two solutions in $\R_+$. Notice that there is no need to investigate the solvability of $u_t(\Lambda_{t,\nu}^{-1}(1/x))=a \pi t$ when $a\geq1/t$, because the angle $u_t<\pi$.

Conversely, suppose now \eqref{equation} has at most two solutions. In order to derive a contradiction, we assume that $\sigma_t\boxtimes \nu$ is not log-unimodal. This means that there exists $a\in (0,1/t)$ such that the equation $xq_t(x)=a$ has at least three distinct solutions $x_1$, $x_2$, and $x_3$. We put $r_i=\Lambda_{t,\nu}^{-1}(1/x_i)$ and deduce that 
the equation $u_t(r_i)=a\pi t$ holds for $i=1,2,3$. It follows that the equation \eqref{equation} has solution $r=r_1,r_2,r_3$, a contradiction. Therefore the measure $\sigma_t\boxtimes \nu$ must be log-unimodal.
\end{proof}

We next address the symmetry and log-unimodality of $\sigma_t\boxtimes \nu$. The identity $(X^{\frac1{2}} Y X^{\frac1{2}})^{-1} = X^{-\frac1{2}} Y^{-1} X^{-\frac1{2}}$ for invertible non-negative self-adjoint operators $X,Y$ affiliated with a finite von Neumann algebra readily shows that $(\mu\boxtimes \nu)^{-1}=\mu^{-1}\boxtimes\nu^{-1}$, so that the free convolution of two multiplicatively symmetric measures is multiplicatively symmetric. In particular, $\sigma_t\boxtimes \nu$ is multiplicatively symmetric whenever $\nu$ is. 

Our first result is a free analogue of Proposition 3.9 in the context of free Brownian motion. Yet, interestingly enough, the proof replies on Proposition 3.9.
\begin{theorem}\label{thm:symmlogunimo}
If $\nu$ is multiplicatively symmetric and log-unimodal, then so are $\sigma_t\boxtimes \nu$ for all $t>0$. In particular, the measure $\sigma_t$ itself is log-unimodal.
\end{theorem}
\begin{proof}
Fix $t>0$. Since $\lambda_{a\pi t}$ and $\nu^{-1}$ are log-unimodal and multiplicatively symmetric, and since the density of $\sigma_t\boxtimes \nu$ is continuous without having plateaux, Proposition 3.9 implies that the measure $\lambda_{a\pi t}\circledast \nu^{-1}$ is log-unimodal and hence \[\#\left\{ r>0: r\cdot \frac{d(\lambda_{a\pi t}\circledast \nu^{-1})}{dx}(r)=\frac{a\pi}{\sin (a\pi t)}\right\} \le 2,\quad a\in (0,1/t).\] By Lemma \ref{lem:freemulti}, the measure $\sigma_t\boxtimes \nu$ is log-unimodal.

If we take $\nu=\delta_1$, then $\sigma_t=\sigma_t\boxtimes \delta_1$ is log-unimodal.
\end{proof}

\begin{problem}
If $\mu$ and $\nu$ are multiplicatively symmetric and log-unimodal on $\R_+$, is the free convolution $\mu\boxtimes \nu$ log-unimodal?\end{problem}

We next show the eventual log-unimodality of $\sigma_t\boxtimes \nu$ when $\nu$ is supported on a suitable finite interval.
\begin{theorem}\label{thm4.4}
Let $\nu$ be a probability measure supported on $[\alpha,\beta]$, where $0<\alpha<\beta$ such that $\beta^4-3\alpha^4<2\alpha^3\beta$. If
\begin{align*}
t\ge D_{\alpha,\beta}=\frac{2\beta^2(\alpha+\beta)^2\pi}{\sqrt{4\alpha^6\beta^2-(3\alpha^4-\beta^4)^2}},
\end{align*} then $\sigma_t\boxtimes \nu$ is log-unimodal.
\end{theorem}

\begin{proof}
We set \[\Theta_R(r)=\frac{\sin R}{R}\int_0^\infty \frac{r\xi}{1+r^2\xi^2-2r\xi \cos R}d\nu(\xi)\] for our purpose. We aim to prove that for each $R\in (0,\pi)$ and $t\ge D_{\alpha,\beta}$, the equation $\Theta_R(r)=\frac{1}{t}$ has at most two solutions. This suffices thanks to Lemma \ref{lem:freemulti}. Note that
\begin{align*}
\Theta_R'(r)&=\frac{\sin R}{R}\int_\alpha^\beta\frac{\xi(1-r^2\xi^2)}{(1-2r\xi \cos R+r^2\xi^2)^2}d\nu(\xi),\\
\Theta_R''(r)&=\frac{\sin R}{R}\int_\alpha^\beta\frac{2\xi^2(r^3\xi^3-3r\xi+2\cos R)}{(1-2r\xi \cos R+r^2\xi^2)^3}d\nu(\xi).
\end{align*}
For all $0<r<1/\beta$ and $\xi\in [\alpha,\beta]$, we have $1-r^2\xi^2>0$ and hence $\Theta_R'(r)>0$. Similarly, we have $\Theta_R'(r)<0$ if $r>1/\alpha$. So, the function $\Theta_R$ is strictly monotonic on $\R_+ \setminus [1/\beta,1/\alpha]$, and therefore the equation $\Theta_R(r)=\frac{1}{t}$ can only have at most two solutions on this complement.

We next consider $r\in [1/\beta,1/\alpha]$ and distinguish two cases according to whether $\cos R\leq \frac{3\alpha^4-\beta^4}{2\alpha^3\beta}$ or not.

Case I: $\cos R\leq  \frac{3\alpha^4-\beta^4}{2\alpha^3\beta}$. For each $\xi\in[\alpha,\beta]$ and $1/\beta<r<1/\alpha$, we have 
\begin{align*}
r^3\xi^3-3r\xi+2\cos R<\frac{\beta^3}{\alpha^3}-\frac{3\alpha}{\beta}+2\cos R=2\left( \frac{\beta^4-3\alpha^4}{2\alpha^3\beta}+\cos R\right)\leq0,
\end{align*}
showing that $\Theta_R''(r)<0$. This means that the function $\Theta_R$ is concave down and the derivative $\Theta_R'$ is strictly decreasing on the interval $(1/\beta,1/\alpha)$. Then the intermediate value theorem shows that the derivative $\Theta_R'$ has a unique zero $p$ in the closed interval $[1/\beta,1/\alpha]$, and hence the critical point $p$ is the unique local maximizer for the function $\Theta_R$ in $[1/\beta,1/\alpha]$. This analysis on the graph of $\Theta_R$ shows that there are at most two solutions for $\Theta_R(r)=\frac{1}{t}$ in this case.

Case II: $\cos R> \frac{3\alpha^4-\beta^4}{2\alpha^3\beta}$. In this case we have
\begin{align*}
\frac{\sin R}{R}>\frac{\sin\left( \cos^{-1}\left(\frac{3\alpha^4-\beta^4}{2\alpha^3\beta}\right)\right)}{\pi}=\frac{1}{\pi}\sqrt{1-\left(\frac{3\alpha^4-\beta^4}{2\alpha^3\beta}\right)^2}, \quad \frac{\pi}{2}\leq R < \pi,
\end{align*} and \begin{align*}
\frac{\sin R}{R}>\frac{2}{\pi}>\frac{1}{\pi}\sqrt{1-\left(\frac{3\alpha^4-\beta^4}{2\alpha^3\beta}\right)^2}, \quad 0<R<\frac{\pi}{2}.
\end{align*} 
On the other hand, observe that
\begin{align*}
1-2r\xi\cos R+r^2\xi^2 \le 1+2r\xi+r^2\xi^2=(1+r\xi)^2\le \left(1+\frac{\beta}{\alpha}\right)^2=\frac{(\alpha+\beta)^2}{\alpha^2},\quad \xi\in [\alpha,\beta]. 
\end{align*}
It follows that
\begin{align*}
\Theta_R(r)>\frac{1}{\pi}\sqrt{1-\left(\frac{3\alpha^4-\beta^4}{2\alpha^3\beta}\right)^2}\times \frac{\alpha^2}{(\alpha+\beta)^2}\times \frac{\alpha}{\beta}=\frac{1}{D_{\alpha,\beta}}\ge \frac{1}{t}
\end{align*}
for all $r \in [1/\beta ,1/\alpha]$ and $R\in (0,\pi)$. This shows that the equation $\Theta_R(r)=\frac{1}{t}$ has no solutions in $[1/\beta ,1/\alpha]$.

In all cases, we have shown that for all $R\in (0,\pi)$, the equation $\Theta_R(r)=\frac{1}{t}$ has at most two solutions if $t\ge D_{\alpha,\beta}$. 
\end{proof}

We follow the ideas in \cite{Hua, HU18} to construct probability measures $\nu$ such that (i) the masses of $\nu$ escape to either $0$ or $+\infty$, and (ii) $\sigma_t\boxtimes \nu$ is not log-unimodal for any $t>0$. 

\begin{theorem}\label{thm:non-log}
Let $\{w_n\}_{n\in \N}$ and $\{a_n\}_{n\in \N}$ be two sequences in $\R_+$ such that
\begin{itemize}
\item $a_n>a_{n+1}$ for all $n\in \N$,
\item $\lim_{k\rightarrow\infty}a_k=0$ and $\lim_{k\rightarrow\infty}\frac{a_ka_{k+1}(a_k+a_{k+1})}{(a_k-a_{k+1})^2}=0$,
\item $\sum_{n=1}^\infty w_n=1$ and $\sum_{n=1}^\infty w_na_n^{-1}<\infty$.
\end{itemize}
Let $\nu=\sum_{n=1}^\infty w_n\delta_{a_n}$. Then both $\sigma_t\boxtimes \nu$ and $\sigma_t\boxtimes \nu^{-1}$ are not log-unimodal for every $t>0$. 
\end{theorem}

\begin{proof}
Define the function
\begin{align*}
f(r)=\int_0^\infty \frac{rx}{(1-rx)^2}d\nu(x)=\sum_{n=1}^\infty \frac{w_na_nr}{(1-a_nr)^2}, \qquad r>0.
\end{align*}
We set \[b_k=\frac{a_{k+1}^{-1}+a_k^{-1}}{2},\quad k\in \N,\] and observe that
\begin{align*}
|1-a_nb_k|\ge \frac{a_n}{2}\left(a_{k+1}^{-1}-a_k^{-1} \right), \qquad n,k\in \N.
\end{align*}
As $k\rightarrow\infty$, we have
\begin{align*}
f(b_k)&=\sum_{n=1}^\infty  \frac{w_na_nb_k}{(1-a_nb_k)^2}\\
&\le \frac{4b_k}{(a_{k+1}^{-1}-a_k^{-1})^2} \sum_{n=1}^\infty w_na_n^{-1}\\
&=\frac{2a_ka_{k+1}(a_k+a_{k+1})}{(a_k-a_{k+1})^2}\sum_{n=1}^\infty w_na_n^{-1}\rightarrow 0.
\end{align*}

Recall that $V_{t,\nu}=\{r\in \R_+:f(r)>1/t\}=\{r\in \R_+:u_t(r)>0\}$. The above limit implies that for each $t>0$, there exists an integer $K(t)>0$ such that $f(b_k)<1/t$ for all $k \ge K(t)$. So, for $k\ge K(t)$, the closure of $V_{t,\nu}$ does not contain $b_k$. Represent the open set $V_{t,\nu}$ as a disjoint union of open intervals, we conclude that the closure $\overline{V_{t,\nu}}$ contains at least two disjoint closed intervals. (None of these two intervals is a singleton set, because the function $u_{t}$ is continuous.) Therefore, under the homeomorphism $1/\Lambda_{t,\nu}$, the support $\supp(\sigma_t\boxtimes \nu)=(1/\Lambda_{t,\nu})(\overline{V_{t,\nu}})$ contains two disjoint closed intervals. It follows that the support $\supp(\log_{*} (\sigma_t\boxtimes \nu))=\log(\supp(\sigma_t\boxtimes \nu))$ of the push-forward measure $\log_{*} (\sigma_t\boxtimes \nu)$ also contains two disjoint closed intervals, so that the continuous density $d\log_{*} (\sigma_t\boxtimes \nu)/dx$ vanishes in between these two intervals. We conclude that $\log_{*} (\sigma_t\boxtimes \nu)$ is not unimodal, that is, $\sigma_t\boxtimes \nu$ is not log-unimodal. 

Since $\sigma_t\boxtimes \nu^{-1}=(\sigma_t\boxtimes \nu)^{-1}$, Proposition 3.3 implies that $\sigma_t\boxtimes \nu^{-1}$ is not log-unimodal for any $t>0$. 
\end{proof}

\begin{remark}
The proof of the preceding result actually shows that $\sigma_t\boxtimes \nu$ is not unimodal.
\end{remark}

\begin{problem}
Can we delete the assumption $\beta^4 -3\alpha^4 < 2\alpha^3\beta$ in Theorem \ref{thm4.4}? More precisely, if $\nu$ is a probability measure supported on $[\alpha,\beta]$ such that $0<\alpha <\beta$, does there exist a number $D\ge0$ (depending on $\nu$) such that $\sigma_t\boxtimes\nu$ is log-unimodal for all $t\ge D$?  
\end{problem}

Now, it is fairly easy to construct counterexamples of Theorem \ref{thm4.4} when $\nu$ has an unbounded support.
\begin{example}\label{ex:unbounded_support}
If $w_n=\frac{945}{\pi^6 n^6}$ and $a_n=n^{-4}$ for all $n\in\N$, then $a_n\le 1$ and
\begin{align*}
\lim_{k\rightarrow\infty}\frac{a_ka_{k+1}(a_k+a_{k+1})}{(a_k-a_{k+1})^2}=0.
\end{align*}
Moreover, we have
\begin{align*}
\sum_{n\ge 1} \frac{w_n}{a_n}=\frac{945}{\pi^6}\sum_{n\ge 1}\frac{1}{n^2}=\frac{945}{\pi^6}\times \frac{\pi^2}{6}=\frac{315}{2\pi^4}<\infty.
\end{align*}
Theorem \ref{thm:non-log} shows that both
\begin{align*}
\sigma_t\boxtimes \left(\frac{945}{\pi^6} \sum_{n=1}^\infty \frac{1}{n^6}\delta_{n^{-4}} \right)\quad \text{and} \quad \sigma_t \boxtimes \left(\frac{945}{\pi^6} \sum_{n=1}^\infty \frac{1}{n^6}\delta_{n^{4}} \right)
\end{align*}
are not log-unimodal for any $t>0$. \end{example}


We conclude this paper with counterexamples of Theorem \ref{thm:symmlogunimo} when $\nu$ is not multiplicatively symmetric.  Recall that a probability measure $\mu$ on $\R_+$ is said to be $\boxtimes$-{\it strongly log-unimodal} if for all log-unimodal distributions $\nu$, the free convolution $\mu\boxtimes\nu$ is log-unimodal. 
\begin{theorem}
There exists $t_0>0$ such that $\sigma_{t}$ is not $\boxtimes$-strongly log-unimodal for any $t \in (0,t_0]$.
\end{theorem}
\begin{proof}
{\bf Step 1.} We show that $\sigma_{t_1}$ is not $\boxtimes$-strongly log-unimodal  for some $t_1>0$, that is, there are $t_1>0$ and a log-unimodal measure $\nu$ such that $\sigma_{t_1}\boxtimes \nu$ is not log-unimodal. 

By Lemma \ref{lem:lambda}, there is a log-unimodal distribution $\nu$ such that $\lambda_1\circledast \nu^{-1}$ is not log-unimodal, where 
\begin{align*}
\lambda_1(dx)=\frac{c_1}{1-2x \cos 1+x^2}dx, \qquad x\in \R_+.
\end{align*}
 It follows that there exists $d>0$ such that
\begin{align*}
\#\left\{ r>0: r\cdot \frac{d(\lambda_1\circledast \nu^{-1})}{dx}(r)=d\right\}\ge 3.
\end{align*}
Now take $a>0$ and $t_1>0$ such that
\begin{align*}
a\pi t_1=1, \qquad \frac{a\pi}{\sin(a\pi t_1)}=\frac{d}{c_1}.
\end{align*}
Indeed, one can choose
\begin{align*}
a=\frac{d\sin 1}{c_1\pi}, \qquad t_1=\frac{c_1}{d\sin 1}.
\end{align*}
In view of Lemma \ref{lem:freemulti} \eqref{item:freemulti3},  the measure $\sigma_{t_1}\boxtimes \nu$ is not log-unimodal.

\vspace{2mm}
\noindent
{\bf Step 2.} Let 
\[
I = \{t \in (0,\infty): \sigma_t \text{~is not $\boxtimes$-strongly log-unimodal}\}. 
\]
We claim that if $t \in I$ and $s \in (0,t)$, then one has either $s \in I$ or $t-s\in I$.  To see this, we take a log-unimodal distribution $\mu$ such that $\sigma_{t}\boxtimes \mu$ is not log-unimodal, and we observe the identity
\begin{equation}\label{LU_NLU}
\sigma_s \boxtimes (\sigma_{t-s} \boxtimes \mu) = \sigma_{t}\boxtimes \mu. 
\end{equation}
If $\sigma_{t-s} \boxtimes \mu$ is not log-unimodal then $t -s \in I$. If  $\sigma_{t-s} \boxtimes \mu$ is log-unimodal then $s \in I$ because of \eqref{LU_NLU}. 

\vspace{2mm}
\noindent
{\bf Step 3.} Let 
\[
t_2 = \max\{t \in [0,t_1]: \sigma_t \text{~is $\boxtimes$-strongly log-unimodal}\}. 
\]
Note that the above subset of $[0,t_1]$ is closed by Lemma \ref{lem:weak}. Also, it is non-empty since $\sigma_0=\delta_1$ is $\boxtimes$-strongly log-unimodal. So the maximum exists. Moreover, since $\sigma_{t_1}$ is not $\boxtimes$-strongly log-unimodal, we have $t_2 \in [0,t_1)$.  

If $t_2=0$ then $(0,t_1] \subset I$ and we may take $t_0=t_1$.

If $t_2>0$ then $t_2+\epsilon \in I$ for every $\epsilon \in (0,t_1-t_2]$. Applying Step 2 to $t = t_2+\epsilon$ and $s=\epsilon$, we conclude that $\epsilon \in I$, since we know $t_2 = t-s \notin I$. This argument shows that $(0,t_1-t_2] \subset I$ and we may take $t_0=t_1-t_2$.  
\end{proof}

\begin{remark}
This result is contrasted with the fact that the log-normal distributions are  all $\circledast$-strongly log-unimodal. 
\end{remark}

\begin{problem} Is it true that $\sigma_{t}$ is not $\boxtimes$-strongly log-unimodal for any $t >0$?
\end{problem}

\subsection*{Acknowledgment}
The first-named author is granted by JSPS kakenhi (B) 19K14546. Support of the third author came from the NSERC Canada Discovery Grant RGPIN-2016-03796. This research is an outcome of Joint Seminar supported by JSPS and CNRS under the Japan-France Research Cooperative Program.

\newpage

\hspace{-6mm}Takahiro Hasebe\\
Department of Mathematics, Hokkaido University,\\
Kita 10, Nishi 8, Kita-Ku, Sapporo, Hokkaido, 060-0810, Japan\\
Email: thasebe@math.sci.hokudai.ac.jp\\
\vspace{1mm}\\
\hspace{6mm}\hspace{-6mm}Yuki Ueda\\
Department of General Science, National Institute of Technology, Ichinoseki College,\\
Takanashi, Hagisho, Ichinoseki, Iwate 021-8511, Japan\\
Email: yuki1114@ichinoseki.ac.jp\\
\vspace{1mm}\\
\hspace{6mm}\hspace{-6mm}Jiun-Chau Wang\\
Department of Mathematics and Statistics, University of Saskatchewan, \\
106 Wiggins Road, Saskatoon, Saskatchewan S7N 5E6, Canada\\
Email: jcwang@math.usask.ca

\end{document}